\documentclass[12pt,leqno]{article}
\usepackage{amssymb, amscd, amsmath, amsthm}
\usepackage{dsfont}
\usepackage[pdftex,colorlinks=true,linkcolor=blue,urlcolor=red,unicode=true,hyperfootnotes=false,bookmarksnumbered]{hyperref}
\usepackage{xcolor}
\allowdisplaybreaks

\def\beq{\begin{equation}}
\def\eeq{\end{equation}}

\theoremstyle{definition}

\newtheorem{observation}{Observation}

\theoremstyle{plain}
\newtheorem{theorem}{Theorem}
\newtheorem{claim}{Claim}

\newtheorem{lemma}{Lemma}

\newtheorem{proposition}{Proposition}

\numberwithin{equation}{section}
\numberwithin{proposition}{section}
\numberwithin{observation}{section}
\numberwithin{definition}{section}
\numberwithin{theorem}{section}
\numberwithin{problem}{section}
\numberwithin{example}{section}
\numberwithin{claim}{section}
\numberwithin{fact}{section}
\numberwithin{lemma}{section}
\numberwithin{conjecture}{section}
\numberwithin{corollary}{section}

\begin{document}

\title{On the maximum number of distinct intersections in an intersecting family}

\author{
Peter Frankl\footnote{R\'enyi Institute, Budapest, Hungary and Moscow Institute of Physics and Technology, Russia, Email: {\tt peter.frankl@gmail.com}},
Sergei Kiselev\footnote{Moscow Institute of Physics and Technology, Email: {\tt kiselev.sg@gmail.com}}, 
Andrey Kupavskii\footnote{G-SCOP, CNRS, University Grenoble-Alpes, France and Moscow Institute of Physics and Technology, Russia; Email: {\tt kupavskii@yandex.ru}.  The authors acknowledge the financial support from the Russian Government in the framework of MegaGrant no 075-15-2019-1926.}}

\date{}
\maketitle

\begin{abstract}
For $n > 2k \geq 4$ we consider intersecting families $\mathcal F$ consisting of $k$-subsets of $\{1, 2, \ldots, n\}$.
Let $\mathcal I(\mathcal F)$ denote the family of all distinct intersections $F \cap F'$, $F \neq F'$ and $F, F'\in \mathcal F$.
Let $\mathcal A$ consist of the $k$-sets $A$ satisfying $|A \cap \{1, 2, 3\}| \geq 2$.
We prove that for $n \geq 50 k^2$ \,$|\mathcal I(\mathcal F)|$ is maximized by $\mathcal A$.
\end{abstract}

\section{Introduction}
\label{sec:1}

Let $n, k$ be positive integers, $n > 2k$.
Let $X = \{1,2,\ldots, n\}$ be the standard $n$-element set and let ${X\choose k}$ be the collection of all its $k$-subsets.
For a family $\mathcal F \subset {X\choose k}$ let $\mathcal I(\mathcal F) := \{F \cap F'\colon F, F' \in \mathcal F, F \neq F'\}$ be the family of all distinct pairwise intersections.
Recall that a family $\mathcal F$ is called \emph{intersecting} if $F \cap F'\neq \emptyset$ for all $F, F'\in \mathcal F$.

One of the cornerstones of extremal set theory is the Erd\H{o}s--Ko--Rado Theorem:

\begin{theorem}[\cite{EKR}]
\label{th:1.1}
Suppose that $\mathcal F \subset {X\choose k}$ is intersecting.
Then
\beq
\label{eq:1.1}
|\mathcal F| \leq {n - 1\choose k - 1}.
\eeq
\end{theorem}

For a fixed element $x\! \in\! X$ define the \emph{full star} $\mathcal S_x$ by $\mathcal S_x\! :=\! \left\{S\! \in\! {X \choose k}\colon x\! \in\! S\right\}$.
Clearly $\mathcal S_x$ is intersecting and it provides equality in \eqref{eq:1.1}.
Subfamilies of $\mathcal S_x$ are called \emph{stars}.
If we permit $n = 2k$ then there are many other intersecting families attaining equality in \eqref{eq:1.1}.
However, Hilton and Milner \cite{HM} proved that for $n > 2k$ the full stars are the only intersecting families with this property.

For a family $\mathcal G \subset 2^X$ define the family of transversals:
$$
\mathcal T(\mathcal G) := \bigl\{ T \subset X\colon |T| \leq k, \,T \cap G \neq \emptyset \ \text{ for all } \ G \in \mathcal G\bigr\}.
$$

With this definition $\mathcal G \subset {X \choose k}$ is intersecting iff $\mathcal G \subset \mathcal T(\mathcal G)$.
For $\mathcal G \subset 2^X$ and $0 \leq \ell \leq n$ define the $\ell$-th level of $\mathcal G$ by $\mathcal G^{(\ell)} := \{G \in \mathcal G, |G| = \ell\}$.

An intersecting family $\mathcal F \subset {X \choose k}$ is called \emph{saturated} if $\mathcal F \cup \{G\}$ ceases to be intersecting for all $G \in {X\choose k} \setminus \mathcal F$.

\setcounter{observation}{1}
\begin{observation}
\label{obs:1.2}
An intersecting family $\mathcal F \subset {X \choose k}$ is \emph{saturated} iff $\mathcal F = \mathcal T(\mathcal F)^{(k)}$.
\end{observation}

The aim of the present paper is to investigate the maximum size of $\mathcal I(\mathcal F)$ over intersecting families $\mathcal F \subset {X \choose k}$.
Since $\mathcal F \subset \widetilde{\mathcal F}$ implies $\mathcal I(\mathcal F) \subset \mathcal I(\widetilde{\mathcal F})$, in the process we may assume that $\mathcal F$ is saturated.

Unless otherwise stated, all considered intersecting families are supposed to be saturated.
We need the following lemma that was essentially proved in \cite{F78}.
In order to state it, recall that a family $\mathcal B$ is called an \emph{antichain} if $B \not\subset B'$ holds for all distinct members $B, B' \in \mathcal B$.
Recall also that an antichain $\{A_1, \ldots, A_p\}$ is called a \emph{sunflower} of size $p$ with center $C$ if
$$
A_i \cap A_j = C \ \text{ for all } \ 1 \leq i < j \leq p.
$$

\setcounter{lemma}{2}
\begin{lemma}
\label{lem:1.3}
Suppose that $\mathcal F \subset {X \choose k}$ is a saturated intersecting family.
Let $\mathcal B = \mathcal B(\mathcal F)$ be the family of minimal (w.r.t. containment) sets in $\mathcal T(\mathcal F)$.
Then
\begin{itemize}
\itemsep=-1pt
\item[{\rm (i)}] $\mathcal B$ is an intersecting antichain,
\item[{\rm (ii)}] $\mathcal F = \left\{H \in {X\choose k}\colon \exists B \in \mathcal B, B \subset H\right\}$,
\item[{\rm (iii)}] $\mathcal B$ contains no sunflower of size $k + 1$.
\end{itemize}
\end{lemma}

The proof is given in the next section.

Define the intersecting family $\mathcal A = \mathcal A(n, k)$ on the ground set $X = \{1, \ldots, n\}$ by
$$
\mathcal A := \left\{A \in {X\choose k}\colon |A \cap \{1,2,3\}| \geq 2\right\}.
$$

The main result of the present paper is

\setcounter{theorem}{3}
\begin{theorem}
\label{th:1.4}
Suppose that $n \geq 50 k^2$, $k \geq 2$ and $\mathcal F \subset {X \choose k}$ is intersecting.
Then
\beq
\label{eq:1.2}
|\mathcal I(\mathcal F)| \leq |\mathcal I(\mathcal A)|.
\eeq
\end{theorem}

Let us note that it is somewhat surprising that the maximum is attained for $\mathcal A$ and not the full star which is much larger.
Let us present the formula for $|\mathcal I(\mathcal A)|$.

\setcounter{proposition}{4}
\begin{proposition}
\beq
\label{eq:1.3}
|\mathcal I(\mathcal A)| = 3 \sum_{0 \leq i \leq k - 2} {n - 3\choose i} + 3 \sum_{0 \leq i \leq k - 3} {n - 3\choose i} + \sum_{0 \leq i \leq k - 4} {n - 3\choose i}.
\eeq
\end{proposition}

\begin{proof}
Let $A, A' \in \mathcal A$.
Then there are seven possibilities for $A \cap A' \cap \{1,2,3\}$,
namely, all non-empty subsets of $\{1, 2, 3\}$.
If $A \cap A' \cap \{1,2,3\} = \{1\}$ then $A \cap \{1,2,3\}$ and $A' \cap \{1,2,3\}$ are $\{1,2\}$ and $\{1, 3\}$ in some order.
Since $n > 2k$ it is easy to see $A \cap A' = \{1\} \cup D$ is possible for all $D \subset \{4, \ldots, n\}$, $|D| \leq k - 2$.

The remaining six cases can be dealt similarly.
\end{proof}

Note that the RHS of \eqref{eq:1.3} can be simplified to
$$
3 \sum_{0 \leq i \leq k - 2} {n - 2\choose i} + \sum_{0 \leq i \leq k - 4} {n - 3\choose i}.
$$
In comparison
$$
\bigl|\mathcal I(\mathcal S_x)\bigr| = \sum_{0 \leq i \leq k - 2} {n - 1\choose i} = 2\sum_{0 \leq i \leq k - 2} {n - 2\choose i} - {n - 2\choose k - 2}.
$$
That is,
\beq
\label{eq:1.4}
\bigl|\mathcal I(\mathcal S_x)\bigr| < \frac23 \bigl|\mathcal I(\mathcal A)\bigr|.
\eeq
Doing more careful calculations, one can replace $\frac23$ with $\frac{n}{3(n-k)}$.

The paper is organized as follows.
In Section \ref{sec:2} we prove Lemma \ref{lem:1.3} and the main lemma (Lemma \ref{lem:2.3}) which provides some upper bounds concerning $\mathcal B(\mathcal F)$.
In Section \ref{sec:3} we prove Theorem \ref{th:1.4}.
In Section \ref{sec:4} we mention some related problems.

\section{Preliminaries and the main lemma}
\label{sec:2}

\begin{proof}[Proof of Lemma \ref{lem:1.3}]
The fact that $\mathcal B$ is an antichain is obvious.
Suppose for contradiction that $B, B' \in \mathcal B$ but $B \cap B' = \emptyset$.
If $|B| = |B'| = k$ then $B, B'\in \mathcal F$ and $B \cap B'\neq \emptyset$ follows.
By symmetry suppose $|B'| < k$.
Now $\mathcal B \subset \mathcal T(\mathcal F)$ implies that $B'\cap F \neq \emptyset$ for all $F \in \mathcal F$.
Choose a $k$-element superset $F'$ of $B'$ with $B \cap F' = \emptyset$.
Since $B' \in \mathcal T(\mathcal F)$, we have  $F'\in \mathcal T(\mathcal F)$.
By Observation \ref{obs:1.2}, $F' \in \mathcal F$.
However, $B \cap F' = \emptyset$ contradicts $B \in \mathcal T(\mathcal F)$.
This proves (i).
Statement (ii) is immediate from the definition of $\mathcal B$.

To prove (iii) suppose for contradiction that $B_0, B_1, \ldots, B_k \in \mathcal B$ form a sunflower with center $C$.
Since $\mathcal B$ is an antichain, $C$ is a proper subset of $B_0$.
Consequently $C \notin \mathcal T(\mathcal F)$.
Thus we may choose $F \in \mathcal F$ satisfying $C \cap F = \emptyset$.

Consider the $k + 1$ pairwise disjoint sets $B_0 \setminus C, \ldots, B_k \setminus C$.
By $|F| = k$ there is some $i$, $0 \leq i \leq k$ with $F \cap B_i \setminus C = \emptyset$.
However this implies $F \cap B_i = \emptyset$, a contradiction.
\end{proof}

In what follows, $\mathcal B:=\mathcal B(\mathcal F)$ is as in Lemma~\ref{lem:1.3}: the family of minimal transversals of $\mathcal F$.
Let us recall the Erd\H{o}s--Rado Sunflower Lemma.

\begin{lemma}[{\cite{ER}}]
\label{lem:2.1}
Let $\ell \geq 1$ be an integer and $\mathcal D \subset {X \choose \ell}$ a
family which contains no sunflower of size $k + 1$.
Then
\beq
\label{eq:2.1}
|\mathcal D| \leq \ell! \,k^{\ell}.
\eeq
\end{lemma}

The following statement is both well-known and easy.

\begin{lemma}
\label{lem:2.2}
Suppose that $\mathcal E \subset {X\choose 2}$ is intersecting, then either $\mathcal E$ is a star or a triangle.
\end{lemma}

We are going to use the standard notation:
for integers $a \leq b$ we set $[a, b] = \{i\colon a \leq i \leq b\}$ and $[n] = [1,n]$. We also write $(x, y)$ instead of $\{x, y\}$  if $x \neq y$.

Based on Lemma \ref{lem:2.1} we could prove \eqref{eq:1.2} for $n > k + 50 k^3$.
To get a quadratic bound we need to improve it under our circumstances.
To state our main lemma we need some more definitions.

Define $t = t(\mathcal B) := \min \bigl\{|B|\colon B \in \mathcal B\bigr\}$.
The \emph{covering number} $\tau(\mathcal B)$ is defined as follows: $\tau(\mathcal B) := \min \bigl\{|T|\colon \,T \cap B \neq \emptyset\ \text{ for all }\ B \in \mathcal B\bigr\}$.
Since $\mathcal F$ is a saturated intersecting family, using (ii) of Lemma~\ref{lem:1.3} we have $\tau(\mathcal B) = t$.

Now we can present our main lemma that is a sharpening of a similar result in \cite{F17}. Put $\mathcal B^{(\le \ell)} := \bigcup_{i=1}^\ell \mathcal B^{(\ell)}$.

\setcounter{lemma}{2}
\begin{lemma}
\label{lem:2.3}
Let $\ell$ be an integer, $k \ge \ell \ge 2$.
Suppose that $\mathcal F \subset {X\choose k}$ is a saturated intersecting family, $\mathcal B = \mathcal B(\mathcal F)$, $t \geq 2$. Assume that $\tau(\mathcal B^{(\le \ell)})\ge 2$.
Then
\beq
\label{eq:2.2}
\bigl|\mathcal B^{(\ell)}\bigr| \leq t \cdot \ell \cdot k^{\ell - 2}.
\eeq
\end{lemma}

\begin{proof}
For the proof we use a branching process.
We need some notation.
During the proof a \emph{sequence} is an ordered sequence of distinct elements of~$X$: $(x_1, x_2, \ldots, x_s)$.
Sequences are denoted by $S, S_1$ etc. and we let $\widehat S$ denote the underlying unordered set: $\widehat S = \{x_1, x_2, \ldots, x_s\}$.

To start the branching process, we fix a set $B_1 \in \mathcal B$ with $|B_1| = t(\mathcal B)$ and for each element $y_1 \in B_1$ we assign weight $t(\mathcal B)^{-1}$ to the sequence $(y_1)$.

At the first stage, we replace each 1-sequence $(y_1)$ with at most $\ell$ 2-sequences: using $\tau \bigl(\mathcal B^{(\le \ell)}\bigr) \geq 2$ we choose an arbitrary $B(y_1) \in \mathcal B$ satisfying $y_1 \notin B(y_1)$, $|B(y_1)| \le \ell$, and assign weight $\left(t(\mathcal B)\cdot |B(y_1)|\right)^{-1} \ge \left(t(\mathcal B)\cdot \ell \right)^{-1}$ to each sequence $(y_1, y_2)$, $y_2 \in B(y_1)$.
Note that the total weight assigned is exactly~$1$.

At each subsequent stage we pick a sequence $S = (x_1, \ldots, x_p)$ with weight~$w(S)$ such that there exists $B \in \mathcal B$ satisfying $\widehat S \cap B = \emptyset$. Then we replace $S$ by the $|B|$ sequences $(x_1, \ldots, x_p, y)$, $y \in B$, and assign weight $\frac{w(S)}{|B|}$ to each of them.

We continue until $\widehat S \cap B \neq \emptyset$ holds for all sequences $S$ and all $B \in \mathcal B$.
Since $X$ is finite, this eventually happens.
Importantly, the total weight assigned is still~$1$.

\setcounter{claim}{3}
\begin{claim}
\label{cl:2.4}
For each $B \in \mathcal B^{(\ell)}$ there is some sequence $S$ with $\widehat S = B$.
\end{claim}

\begin{proof}
Let us suppose the contrary. Since $\mathcal B$ is intersecting, a sequence with $\widehat S = B$ is not getting replaced by a longer sequence during the process. Let $S = (x_1, \ldots, x_p)$ be a sequence of maximal length that occurred at some stage of the branching process satisfying $\widehat S \subsetneqq B$.
Since $\widehat S$ is a proper subset of $B$, $\widehat S \cap B' = \emptyset$ for some $B'\in \mathcal B$.
Thus at some point we picked $S$ and chose some $\widetilde B \in \mathcal B$ disjoint to it.
Since $\mathcal B$ is intersecting, $B \cap \widetilde B \neq \emptyset$.
Consequently, for each $y \in B \cap \widetilde B$ the sequence $(x_1, x_2, \ldots, x_p, y)$ occurred in the branching process.
This contradicts the maximality of $p$.
\end{proof}

Let us check the weight assigned to $S$ with $|\widehat S| = \ell$.
It is at least $1\bigm/ \bigl(t(\mathcal B)\ell k^{\ell - 2}\bigr)$.
Since the total weight is $1$, \eqref{eq:2.2} follows.
\end{proof}

We should remark that the same $B \in \mathcal B^{(\ell)}$ might occur as $\widehat S$ for several sequences $S$ and for many sequences
$|\widehat S| \neq \ell$ might hold.
This shows that there might be room for considerable improvement.

Let us mention that if $\mathcal F \subset {X \choose k}$ is a saturated intersecting family with $\tau(\mathcal F) = k$ then $\mathcal B(\mathcal F) = \mathcal F$ and \eqref{eq:2.2} reduces to $|\mathcal F| \leq k^k$, an important classical result of Erd\H{o}s and Lov\'asz \cite{EL}.

\section{The proof of Theorem \ref{th:1.4}}
\label{sec:3}

Since the case $k = 2$ trivially follows from Lemma~\ref{lem:2.2}, we assume $k \geq 3$. Take any saturated intersecting $\mathcal F$ and let $\mathcal B = \mathcal B(\mathcal F)$.
First recall that for the full star $\mathcal S_x$,
\beq
\label{eq:3.1}
\bigl|\mathcal I(\mathcal S_x)\bigr| = \sum_{0 \leq \ell \leq k - 2} {n - 1\choose \ell},\eeq
 which is less than $|\mathcal I(\mathcal A)|$, and so we
may assume that $\mathcal F$ is not the full star, i.e., $\mathcal B^{(1)} = \emptyset$.

Recall that $t = \min \{|B|\colon B \in \mathcal B\}$.
Let us first present two simple inequalities for sums of binomial coefficients that we need in the sequel.
$$
{n - a\choose i} \Bigm/{n - a\choose i - 1} = \frac{n - a - i + 1}{i} \geq \frac{n - k}{k - 1}
$$
$\text{ for } k > i > 0, \ a \geq 0 \text{ and }
i + a \leq k + 1$.
Thus for every $1 \leq s \leq k - 1$
\beq
\label{eq:3.1uj}
\sum_{0 \leq i \leq s} {n - 2\choose i} \geq \frac{n - k}{k - 1} \sum_{0 \leq i \leq s - 1} {n - 2\choose i}.
\eeq
$$
{n - 2\choose i} \Bigm/ {n \choose i} = \frac{(n - i)(n - i - 1)}{n(n - 1)} > \left(1 - \frac{k}{n}\right)^2 \ \text{ for  } \ 1 < i < k.
$$
Thus
\beq
\label{eq:3.2}
\sum_{0 \leq i \leq s} {n - 2\choose i} \geq \left(1 - \frac{k}{n}\right)^2 \sum_{0 \leq i \leq s} {n \choose i} \ \text{ for } \ 1 < s < k.
\eeq
Let us partition $\mathcal F$ into $\mathcal F^{(t)} \cup \ldots \cup \mathcal F^{(k)}$ where $F \in \mathcal F^{(\ell)}$ if $\ell = \max \{ |B|\colon \, B \in \mathcal B, B \subset F\}$.

Set $\mathcal I_\ell = \left\{ F \cap F'\colon F \in \mathcal F^{(\ell)}, F' \in \mathcal F^{(t)} \cup \ldots \cup \mathcal F^{(\ell)}\right\}$.
In human language, if $F \in \mathcal F^{(\ell)}$, $F' \in \mathcal F^{(\ell')}$ then we put
$F \cap F'$ into $\mathcal I_\ell$ iff $\ell'\leq \ell$.
It should be clear that
$$
|\mathcal I(\mathcal F)| \leq \sum_{t \leq \ell \leq k} |\mathcal I_\ell|.
$$
The point is that for $F \in \mathcal F^{(\ell)}$ and $B \subset F$, $B \in \mathcal B^{(\ell)}$, for an arbitrary $F'\in \mathcal F$, $F \cap F'$ is partitioned as
$$
F \cap F' = (B \cap F') \cup \bigl((F\setminus B) \cap F'\bigr).
$$
Here there are at most $2^\ell - 1$ possibilities for $B \cap F'$ and $(F \setminus B) \cap F'$ is a subset of $X$ of size at most $k - \ell$.
This proves

\begin{lemma}
\label{lem:3.1}
For any $t \leq \ell \leq k$ such that $\tau(\mathcal B^{(\le \ell)})\ge 2$ we have
\beq
\label{eq:3.3}
\bigl|\mathcal I_\ell\bigr| \leq (2^\ell - 1) \bigl|\mathcal B^{(\ell)}\bigr| \sum_{0 \leq i \leq k - \ell} {n\choose i}
< 2^\ell \cdot \ell^2 k^{\ell - 2} \sum_{0 \leq i \leq k - \ell} {n \choose i} =: f(n, k, \ell).
\eeq
\end{lemma}

Note that if $\tau(\mathcal B^{(2)}) = 2$ then $\mathcal F$ coincides with $\mathcal A$ and we have nothing to prove. Let $\alpha$ be the smallest integer such that $\tau(\mathcal B^{(\le \alpha)})\ge 2$. We have $\alpha\ge 3$. 
The family $\mathcal F':=\bigcup_{i=1}^{\alpha-1}\mathcal F^{(i)}$ is a trivial intersecting family, and thus
\beq\label{eq:3.333}
\left|\bigcup_{i=1}^{\alpha-1}\mathcal I_{i}\right| \le |\mathcal I(\mathcal S_x)|.
\eeq
On the other hand, using \eqref{eq:3.1uj} it is clear that for $\ell\ge 2$
$$
f(n, k, \ell) \bigm/ f(n, k, \ell + 1) > 
\frac{(n - k)\ell^2}{2k^2(\ell+1)^2} \geq
6 \ \ \text{ for } \ n \geq 50 k^2.
$$
Hence
\beq
\label{eq:3.4}
\sum_{\ell = \alpha}^{k} \bigl|\mathcal I_\ell\bigr| < \sum_{3 \leq \ell \leq k} f(n, k, \ell) < \frac{6}{5} f(n, k, 3).
\eeq

Summing the right hand sides of \eqref{eq:3.333} and \eqref{eq:3.4}, we get that 
\begin{align*}
|\mathcal I(\mathcal F)|\le &\,
\sum_{0 \leq i \leq k - 2} {n - 1\choose i} + \frac{432k}{5} \sum_{0\le i\le k-3}{n\choose i} \\
<&\,
{n-2\choose k-2} + \left( \frac{432k}{5} + 2 \right)\sum_{0\le i\le k-3}{n\choose i} \\
\overset{\eqref{eq:3.1uj}}{\le}&\, 
{n-2\choose k-2} + 90\frac{k(k-1)}{n-k}\sum_{0\le i\le k-2}{n\choose i}\le
{n-2\choose k-2}+1.8\sum_{0\le i\le k-2}{n\choose i} \\
\overset{\eqref{eq:3.2}}{<} &\, {n-2\choose k-2}+2\sum_{0\le i\le k-2}{n-2\choose i} < 
|\mathcal I(\mathcal A)|.
\qed \popQED
\end{align*}

\section{Concluding remarks}
\label{sec:4}

Let $n_0(k)$ be the smallest integer such that Theorem \ref{th:1.4} is true for $n \geq n_0(k)$.
We proved $n_0(k) \leq 50 k^2$.
One can improve on the constant $50$ by being more careful in the analysis.
The following example shows that $n_0(k) \ge (3 - \varepsilon)k$.

For $1 \le p \le k$, $n > 2k$, define the family $\mathcal B_p(n, k)$ by
$$
\mathcal B_p(n, k) := \left\{ A\in\binom{[n]}{k}\colon |A\cap [2p - 1]| \ge p \right\}.
$$
Note that $\mathcal S_1 = \mathcal B_1(n, k)$ and $\mathcal A = \mathcal B_2(n, k)$. It is easy to verify that
\begin{align*}
|\mathcal I(\mathcal B_p(n, k))| = &\,
\sum_{i = 1}^{p-1} \binom{2p - 1}{i} \sum_{j=0}^{k-p}\binom{n - 2p + 1}{j} \\
+ &\,
\sum_{i = p}^{2p-1} \binom{2p - 1}{i} \sum_{j=0}^{ k-i-1}\binom{n - 2p + 1}{j}.
\end{align*}
By doing some calculations, one can see that $|\mathcal I(\mathcal B_3(n, k))| > |\mathcal I(\mathcal B_2(n, k))|$ for $n < (3 - \varepsilon) k$.
It would be interesting do decide, whether for $n > (1 + \varepsilon)k$ the maximum is always attained on one of the families $\mathcal B_p(n, k)$.

Note that for $n = 2k$, $k \ge 14$, it is possible to construct an intersecting family $\mathcal F$ with $|\mathcal I(\mathcal F)| = \sum_{i=0}^{k - 1}\binom{n}{i}$ using an argument from \cite{FKKP}.
We say that a family $\mathcal F$ \emph{almost shatters} a set $X\subset[n]$ if for any $A\subset X$, $A\notin\{\emptyset, X\}$, there is $F\in\mathcal F$ such that $F\cap X = A$. Take a random intersecting family $\mathcal F$ by picking a $k$-set from each pair $(A, [n]\setminus A)$ independently at random. In \cite[Theorem 7]{FKKP} it is proved, that with positive probability $\mathcal F$ almost shatters every $X\in\binom{[2k]}{k}$. Fix such a family $\mathcal F$; then, by applying the almost shattering property two times, it is easy to show that, for each $I\subset[n]$, $1\le|I|<k$, there are two sets $F_1, F_2\in\mathcal F$, such that $I\subset F_1$ and $F_1 \cap F_2 = I$.

\medskip

Another natural problem is to consider $\widetilde{\mathcal I}(\mathcal F) = \{F \cap F'\colon F, F' \in \mathcal F\} = \mathcal I(\mathcal F) \cup \mathcal F$. Essentially the same proof shows that $\widetilde{\mathcal I}(\mathcal F)$ is maximised by $\mathcal F = \mathcal S_x$ for $n \ge 50 k^2$ and one can verify that $|\mathcal I(\mathcal B_2(n, k))| > |\mathcal I(\mathcal S_x)|$ for $n < (5 - \varepsilon) k$.

\medskip

In \cite{F20} the analogous problem for the number of distinct differences $F \setminus F'$ was considered.
Improving those results in \cite{FKK}, we proved that for $n > 50k \cdot \log k$ the maximum is attained for the full star, $\mathcal S_x$.
We showed also that it is no longer true for $n = ck$, $2 \leq c < 4$, $k > k_0(c)$.

The methods used in \cite{FKK} are completely different.

\end{document}